\documentclass[11pt, a4paper]{article}
\usepackage{amssymb,amsmath,amsthm}
\usepackage{fancyhdr}
\usepackage[british]{babel}
\usepackage{enumitem}

\usepackage{tikz}
\usetikzlibrary{decorations.pathreplacing}

\usepackage{hyperref}
\usepackage[margin=2.5cm]{geometry}

\newtheorem{theorem}{Theorem}[section]
\newtheorem{prop}[theorem]{Proposition}
\newtheorem{lemma}[theorem]{Lemma}
\newtheorem{cor}[theorem]{Corollary}

\theoremstyle{definition}

\newtheorem{question}{Question}

\newtheorem{expl}{Example}

\newtheorem{defn}[theorem]{Definition}
\newtheorem{rmk}{Remark}

\newtheorem{claim}[theorem]{Claim}

\newcommand*{\abs}[1]{\lvert #1\rvert}
\newcommand*{\floor}[1]{\lfloor #1\rfloor}

\newcommand{\ep}{\varepsilon}

\newcommand{\ex}{\mathrm{ex}}
\newcommand{\eg}{e.g.\ }
\newcommand{\ie}{i.e.\ }

\title{Crux and long cycles in graphs}

\author{
	John Haslegrave\thanks{Mathematical Institute, University of Oxford, UK.  Email: {\tt j.haslegrave@cantab.net}. J.Ha.\ was supported by the UK Research and Innovation Future Leaders Fellowship MR/S016325/1.}
	\and
	Jie Hu\thanks{Laboratoire Interdisciplinaire des Sciences du Num\'erique, Universit\'e Paris-Saclay, France. Email: {\tt hujie@lri.fr}.}
	\and
	Jaehoon Kim
	\thanks{Department of Mathematical Sciences, KAIST, South Korea. Email: {\tt jaehoon.kim@kaist.ac.kr}. J.K. was supported by the POSCO Science Fellowship of POSCO TJ Park Foundation and by the KAIX Challenge program of KAIST Advanced Institute for Science-X.}
	\and
	Hong Liu\thanks{Extremal Combinatorics and Probability Group (ECOPRO), Institute for Basic Science (IBS), Daejeon, South Korea, Email: {\tt hongliu@ibs.re.kr}. H.L. was supported by the Institute for Basic Science (IBS-R029-C4) and the UK Research and Innovation Future Leaders Fellowship MR/S016325/1.}
	\and
	Bingyu Luan\thanks{School of Mathematics and Data Science Institute, Shandong University, China. Email: {\tt  byluan@mail.sdu.edu.cn, ghwang@sdu.edu.cn}. B.L. and G.W. were supported by Natural Science Foundation of China (11871311) and seed fund program for international research cooperation of Shandong University.}
	\and
	Guanghui Wang\footnotemark[5]
}

\begin{document}
\maketitle

\begin{abstract}
  We introduce a notion of the \emph{crux} of a graph $G$, measuring the order of a smallest dense subgraph in $G$. This simple-looking notion leads to some generalisations of known results about cycles, offering an interesting paradigm of `replacing average degree by crux'. In particular, we prove that \emph{every} graph contains a cycle of length linear in its crux.

  Long proved that every subgraph of a hypercube $Q^m$ (resp.\ discrete torus $C_3^m$) with average degree $d$ contains a path of length $2^{d/2}$  (resp.\ $2^{d/4}$), and conjectured that there should be a path of length $2^{d}-1$ (resp.\ $3^{d/2}-1$). As a corollary of our result, together with isoperimetric inequalities, we close these exponential gaps giving asymptotically optimal bounds on long paths in hypercubes, discrete tori, and more generally Hamming graphs.

We also consider  random subgraphs of $C_4$-free graphs and hypercubes, proving near optimal lower bounds on the lengths of long cycles.

\end{abstract}

\section{Introduction}
The study on the existence of long cycles in graphs has a rich history. A celebrated result of Dirac \cite{Dirac} states that every graph $G$ on $n\ge 3$ vertices with minimum degree $\delta(G)\ge n/2$ contains a Hamiltonian cycle. However, any graph satisfying Dirac's condition is dense, having $\Theta(n^2)$ edges.
A natural line of work is to consider how long a cycle we can ensure in a well-connected \emph{sparse} graph.

\subsection{Motivations}
A folklore result on cycles is that any cyclic graph $G$ contains a cycle of length linear in its average degree, i.e.~$\Omega(d(G))$. Indeed, remove low-degree vertices to obtain a subgraph $H$ with $\delta(H)\geq d(G)/2$ and then greedily extend a path to find a cycle in $H$ of length at least $\delta(H)+1$.  This linear in average degree lower bound is the best we could hope for, as the graph $G$ might be a disjoint union of cliques. It seems intuitive that better bounds can be obtained if we step away from such examples. This motivates the following notion of the \emph{crux} of a graph; it measures the order of the smallest subgraph of $G$ which retains a positive fraction of the average degree of $G$.

\begin{defn}[Crux]
	For a constant $\alpha\in(0,1)$, a subgraph $H\subseteq G$ is an $\alpha$-\emph{crux} if $d(H)\ge \alpha\cdot d(G)$. Define the $\alpha$-\emph{crux function}, $c_{\alpha}(G)$, of $G$ to be the order of a minimum $\alpha$-crux in $G$, that is,
	\[c_{\alpha}(G)=\min\{ \abs{H}: H\subseteq G\text{ and } d(H)\ge \alpha\cdot d(G)\}.\]
\end{defn}
Note that trivially we have $c_\alpha(G)>\alpha\cdot d(G)$, $c_\alpha(G)\geq c_{\alpha'}(G)$ for $\alpha\geq \alpha'$, and that if $H\subseteq G$ with $d(H)\geq d(G)/2$ then $c_{2\alpha}(H)\geq c_\alpha(G)$.

\medskip

In this paper, we investigate the following `replacing average degree by crux' heuristic.

\begin{question}\label{ques: meta-ques}
	Suppose we have a result guaranteeing the existence of a certain substructure whose size is a function of $d(G)$ (or $\delta(G)$). Under what circumstances can we replace $d(G)$ (or $\delta(G)$) with $c_\alpha(G)$?
\end{question}

Positive instances for the above question would lead to improvements on embedding problems for graph classes whose crux size is much larger than their average degree.

\setcounter{expl}{1}
\begin{expl}
	There are many natural classes of graphs having $c_\alpha(G)$ much larger than $d(G)$.
	Some specific classes are graphs with geometric structure, such as hypercubes $Q^m$ and Hamming graphs $H(m,r)$, which are  Cartesian products of $m$ complete graphs $K_r$:
	\begin{equation}\label{eq-iso}
		c_\alpha(Q^m)\ge 2^{\alpha m},\quad\quad c_\alpha(H(m,r))\ge r^{\alpha m};\footnote{See Propositions~\ref{prop: edge-isoperimetry} and~\ref{prop: discrete-torus-edge-isoperimetry-2}.}
	\end{equation}
	$K_{s,t}$-free graphs $G$ with $s,t\ge 2$, which satisfy $c_\alpha(G)\ge \frac{(\alpha d(G))^{s/(s-1)}}{2t}$ (since, by a result of K\H{o}v\'ari, S\'os and Tur\'an \cite{K-S-T}, we have $2t|H|\ge (d(H))^{s/(s-1)}$ for every $K_{s,t}$-free graph $H$ with $s,t\ge 2$); and blow-ups of $r$-regular expander graphs for a constant $r$.
\end{expl}

Let us first see an example of a positive answer to Question~\ref{ques: meta-ques}.

\begin{expl}
	A classical result of Koml\'os and Szemer\'edi \cite{K-Sz-2} and of Bollob\'as and Thomason \cite{B-Th} says that every graph $G$ contains a topological clique of order $\Omega(\sqrt{d(G)})$. This result is tight by the example of disjoint union of complete bipartite graphs. However, in upcoming work~\cite{IKL}, it is proved that every graph $G$ contains a topological clique of order $\Omega(\sqrt{c_\alpha(G)}/(\log c_\alpha(G))^{1/2+o(1)})$. Since $c_\alpha(G)=\Omega(d(G)^2)$ when $G$ is a $C_4$-free graph, this implies Mader's conjecture that $C_4$-free graphs contains topological cliques of order linear in its average degree, up to polylogarithmic factors \cite{Mader}. (Actually, Liu and Montgomery \cite{LM1} have demonstrated that Mader's conjecture is true using different tools.)
\end{expl}

From this example, we suspect that the following can be a possible philosophical answer to Question~\ref{ques: meta-ques}: replacement is possible when when `spatial constraints' (not having enough vertices) rather than `degree constraints' (not having
a vertex of sufficiently large degree) are the main obstruction to finding the desired substructure. So, for instance, crux is helpful for finding subdivisions of long cycles or large
complete graphs but not of wheels. Indeed, when finding cycles or clique subdivisions, the average degree $d(G)$ may act as a  `spatial constraint'. In other words, the extremal examples in these cases are either disjoint union of cliques $K_d$ or complete bipartite graphs $K_{d,d}$, hence there is not enough `space' to find $C_{\omega(d)}$ or $K_{\omega(\sqrt{d})}$-subdivision. However, a larger value of $c_{\alpha}(G)$ lifts up this `spatial constraint' so we can improve the result (see Theorem~\ref{thm: crux} and Example~C). On the other hand, if $d(G)$ acts as a strong `degree constraint', then this improvement might not be possible.
For an example, let $W_t$ be a wheel, which is obtained from a cycle $C_t$ by adding a new vertex adjacent to all other vertices. Indeed, using the fact that we can always find a subgraph of connectivity linear in $d(G)$ and Menger's theorem, one can always find a $W_{\Omega(d)}$-subdivision in a graph with average degree $d$. However, in this problem, as the graph $G$ could be almost regular, imposing a large crux size on $G$ does not help us to find a subdivision of $W_{\omega(d)}$. This is because $d(G)$ acts as an essential degree constraint rather than a spatial constraint. In this spirit, cycles are perfect examples to investigate Question~\ref{ques: meta-ques}, because `spatial constraints' are much more important than `degree constraints' in finding cycles as every vertex in a cycle has degree only two.

Let us consider another motivating question regarding cycles in expanders, i.e.~graphs in which vertex subsets expand to large neighbourhoods. Originally introduced for network design, expanders, apart from being a central notion in graph theory, also have close interplay with other areas of mathematics and theoretical computer science, see e.g. the comprehensive survey of Hoory, Linial and Wigderson~\cite{survey:HLW}. The type of expanders hitherto studied usually have constant expansion, \ie are linear expanders. We consider here instead expanders with sublinear expansion, introduced by Koml\'os and Szemer\'edi in the 90s \cite{K-Sz-1,K-Sz-2}. We defer the formal definition of sublinear expanders to Section~\ref{sec:sublinear-expander}. This notion of sublinear expanders has proved to be a powerful tool for embedding sparse graphs, playing an essential role in the recent resolutions of several long-standing conjectures that were previously out of reach, see \eg \cite{FKKL,H-K-L,IKL,KLShS17,LM1,LM2,LWY}.
It would therefore be useful to study these sublinear expanders.

Cycle lengths in linear expanders have been well studied, see \eg \cite{FK,Kri2}. In particular, Krivelevich~\cite{Kri2} proved that every linear expander contains a cycle of length linear in its order. What about sublinear expanders? Note that we \emph{cannot} necessarily find a linear-sized cycle, unlike the linear expander case, as the following example shows.

\begin{expl}\label{ex:imbalanced-bip-graph}
The imbalanced complete bipartite graph $K_{n,\frac{n}{\log^2n}}$ is a sublinear expander, but any cycle must take half its vertices from the smaller part, and consequently has length sublinear in the total number of vertices.
\end{expl}

However, in the case of $K_{n,\frac{n}{\log^2n}}$ we can instead consider a subexpander $H=K_{n',n'}$, where $n'=\frac{n}{\log^2n}$, which has average degree about half of $K_{n,\frac{n}{\log^2n}}$. Now this subexpander $H$ does have a cycle of length linear in the order of $H$.
Does such a phenomenon always occur? That is, is it true that if we cannot find a linear-sized cycle in a sublinear expander $G$, then we can find within $G$ a subgraph
$H$, with about the same average degree as $G$, that has a cycle of length linear in the order of $H$? We shall see shortly that this is indeed the case.

\subsection{Crux and cycles}
Our first result finds a cycle of length linear in the crux size in generic graphs, extending the aforementioned folklore result of cycles linear in average degree and giving an instance of a positive answer to Question~\ref{ques: meta-ques}.

\begin{theorem}\label{thm: crux}
Let $0<\alpha<1$. Then every graph $G$
 contains a cycle of length at least
$$\frac{1-\alpha}{16000}\cdot c_{\alpha}(G),$$
provided that a single edge is considered to be a cycle of length one.
\end{theorem}

It is worth mentioning that the above statement for $\alpha<1/2$ can be deduced using a variant of the classical P\'osa's lemma~\cite{Posa} that if sets up to size $k$ expands linearly, then there is a cycle of length $\Omega(k)$. To see this, first pass to a subgraph $H$ with $\delta(H)\ge d(G)/2$; clearly $\abs{H}\geq c_{1/2}(G)\geq c_\alpha(G)$. Then every set $X\subseteq V(H)$ of size $O(c_{\alpha}(G))$ must expand linearly, for otherwise $H[X\cup N_H(X)]$ has average degree almost $d/2$ while having smaller order than $c_\alpha(G)$, a contradiction. Such argument, however, cannot push $\alpha$ beyond $1/2$ as we cannot guarantee the minimum degree of a graph to be larger than half of its average degree, see the bipartite graph in Example \ref{ex:imbalanced-bip-graph}.

\setcounter{rmk}{4}
\begin{rmk}	
	The value of Theorem~\ref{thm: crux} is that we can take $\alpha=1-o(1)$, which is needed to close the exponential gaps in the applications below, see~Corollaries~\ref{thm: weaker-version-Long's-question} and~\ref{thm: weaker-version-Long's-question2}. The idea to get the whole range $0<\alpha<1$ is to pass to an expander subgraph with different expansion threshold $t$ to have better expansions for large sets.
\end{rmk}

We have the following corollary on cycles in sublinear expanders. The bipartite graph in Example \ref{ex:imbalanced-bip-graph}, which is an $(\ep,t)$-expander for any $0<\ep\le 1$ and $t=15$, shows that both terms in the bound below are best possible up to multiplicative constants.

\begin{cor}\label{ques: cycle-expander}
	Let $0<\alpha<1$, $0<\ep\le \frac{1-\alpha}{500}$, $t\geq 1$ and suppose $n\geq 150t$. Then every $n$-vertex $(\ep,t)$-expander $G$ contains a cycle of length
	$$\max\Big\{~ \frac{\ep}{32}c_{\alpha}(G)~, ~~ \frac{\ep n}{1200\log^{2}n}~\Big\}.$$
\end{cor}

\subsection{Application to Long's conjecture}
Long~\cite[Conjecture 8.9]{LONG} conjectured that any subgraph of the hypercube $Q^m$ that has average degree $d$ contains a path of length at least $2^{d}-1$. He obtained a weaker bound and showed that there is a path of length at least $2^{d/2}-1$, by passing to a subgraph of minimum degree at least $d/2$. A similar conjecture for discrete tori $C^m_3$ was made in the same paper. Long proved that every subgraph of $C^m_3$ that has average degree at least $d$ contains a path of length at least $2^{d/4}-1$, and he conjectured~\cite[Conjecture 8.3]{LONG} that the correct bound should be $3^{d/2}-1$. Both conjectures, if true, would be best possible by considering sub-hypercubes or sub-torus.

Using Theorem~\ref{thm: crux} and the isoperimetric inequalities~\eqref{eq-iso}, we immediately close the above exponential gaps and settle both conjectures asymptotically. It would be interesting to see if stability methods can be combined to obtain exact results.

\begin{cor}\label{thm: weaker-version-Long's-question}
	Every subgraph of the hypercube with average degree $d$ contains a cycle of length
	$$2^{d-o(d)}.$$
\end{cor}
\begin{proof}
	Fix arbitrary $0<\ep<1$ and let $H\subseteq Q^m$ be a subgraph with $d(H)= d$. By the definition of crux and~\eqref{eq-iso}, we have $c_{1-\ep}(H)\ge c_{(1-\ep)\frac{d}{m}}(Q^m)\ge 2^{(1-\ep)d}$. Then by Theorem~\ref{thm: crux}, $H$ contains a cycle of length at least $\frac{\ep}{16000}2^{(1-\ep)d}$ as desired.
\end{proof}

The same proof applies also to Hamming graphs. The case $r=3$ below covers discrete tori.
\begin{cor}\label{thm: weaker-version-Long's-question2}
	Every subgraph of the Hamming graph $H(m,r)$ with average degree $d$ contains a cycle of length
	$$r^{\frac{d}{r-1}-o(d)}.$$
\end{cor}


\subsection{Random subgraphs of a given graph}
Our next instances of positive answers to Question~\ref{ques: meta-ques} concern long cycles in random subgraphs of a given graph.
For a given finite graph $G$ and a real $p \in [0, 1]$, let $G_p$ be a random subgraph of $G$ obtained by taking each edge independently with probability $p$. Analysis of $G_p$ can be used to demonstrate the robustness of a graph $G$ with respect to a graph property $\mathcal{P}$, see \eg \cite{K-L-S-0,K-L-S}. If $G$ is the complete graph $K_n$, then $G_p$ is simply the Erd\H os--R\'enyi binomial random graph $G(n, p)$. We say an event happens \emph{asymptotically almost surely} (a.a.s.)\ or \emph{with high probability} (w.h.p.)\ in $G(n,p)$ if its probability tends to $1$ as $n\to\infty$.

Long paths, cycles and Hamiltonicity in $G(n, p)$ have been intensively studied, see \eg \cite{AKS1,AKS3,B1,BFF,Frieze,K-Sz-0,Kri0,LS,Posa}. In particular, Frieze \cite{Frieze} proved that for large $C$, w.h.p.\ $G(n,C/n)$ has a cycle of length at least $(1-(1-o_C(1))Ce^{-C})n$. Krivelevich, Lee and Sudakov \cite{K-L-S} extended these classical results of long paths and cycles in $G(n,p)$ to random subgraphs $G_p$, where $G$ has large minimum degree. For long cycles, they proved that given a graph $G$ with minimum degree $k$, if $pk\to \infty$, then w.h.p.\ $G_p$ contains a cycle of length at least $(1-o(1))k$.
Riordan \cite{R} subsequently gave a shorter proof, and Ehard and Joos~\cite{EJ} further improved the error term. Krivelevich and Samotij~\cite{KS14} later considered graphs without a fixed bipartite subgraph $H$; in the case of $C_4$-free $G$ with $\delta(G)\ge k$, they showed that for $p=\frac{1+\ep}{k}$, w.h.p. $G_p$ contains a cycle of length $\Omega_{\ep}(k^2)$. We give a short proof for random subgraphs of $C_4$-free graphs with $p=\omega(\frac{1}{k})$.
Note that the constant 1 below is best possible, as there are $C_4$-free graphs with minimum degree $k$ and order $(1+o(1))k^2$, see the $C_4$-free construction due to Erd\H{o}s, R\'enyi and S\'os \cite{Erdos}.

\begin{theorem}\label{thm:C4-free}
Suppose that $pk\rightarrow \infty$ as $k\rightarrow \infty$. Let $G$ be a $C_4$-free graph with minimum degree $k$.
Then w.h.p.\ $G_p$ contains a cycle of length at least $(1-o(1))k^2$.
\end{theorem}

Random subgraphs of the hypercube are also well studied, see e.g.~\cite{AKS,CDGKO, JEMS}. For hypercubes, we obtain the following near linear bound. It would be interesting to prove a linear bound. While this paper was being prepared, Erde, Kang and Krivelevich~\cite{EKK} proved Theorem \ref{thm: long-cycle-in-hypercube2} with a better error term $\Omega(\frac{2^{m}}{m^{3}\log^3m})$.

\begin{theorem}\label{thm: long-cycle-in-hypercube2}
	Let $Q^m$ be the $m$-dimensional hypercube. If $p=\frac{1+\ep}{m}$, where $\ep>0$, then w.h.p.\ $Q^m_p$ contains a cycle of length $\frac{2^{m}}{4 m^{32}} =2^{(1-o(1))m}$.
\end{theorem}

\medskip

\noindent\textbf{Organisation.} The rest of the paper is organised as follows. Section \ref{sec:prelim} contains some necessary tools needed in our proofs. In Section \ref{sec:sublin}, we give the proofs of Theorem~\ref{thm: crux} and Corollary~\ref{ques: cycle-expander}. We prove Theorems~\ref{thm:C4-free} and~\ref{thm: long-cycle-in-hypercube2} in Section \ref{sec:cube}. Concluding remarks are given in Section \ref{sec:conc}.

\section{Preliminaries}\label{sec:prelim}
For $a,b \in \mathbb{N}$ with $a<b$, let $[a]:= \{1, \ldots, a\}$ and $[a,b]:=\{a,a+1,\ldots,b\}$. 
We use the standard Landau symbols $O, \Omega, \Theta, o, \omega$ to denote the asymptotic behavior of functions. If a hidden constant depends on some other
constant $\ep$, we write $\Omega_{\ep}(\cdot)$. In many cases, we treat large numbers as if they were integers, by omitting floors and ceilings if it does not affect the argument. We write $\log$ for the natural logarithm.

Given a graph $G$, denote its order and size by $\abs{G}$ and $e(G)$ respectively, and its average degree $2e(G)/\abs{G}$ by $d(G)$. For a vertex subset $U\subseteq V(G)$, write $N_G(U):= \{v\in V(G)\setminus U: v \text{ has a neighbour in } U\}$ for its external neighbourhood; write $\partial U$ for the edge boundary of $U$, that is, $E_G(U,V(G)\setminus U)$; and write $G-U=G[V(G)\setminus U]$ for the subgraph induced on $V(G)\setminus U$.

\subsection{Depth First Search}
We will need Depth First Search (DFS), which is a graph exploration algorithm that visits all the vertices of an input graph. It may be summarised as follows. We maintain a searching stack $S$ (initially empty), a set of unexplored vertices $U$ (initially $V(G)$), and a set of explored vertices $X$ (initially empty), as well as a spanning subgraph $F$, initially empty. At each step, if $S$ is empty but $U$ is not, remove an arbitrary vertex of $U$ and push it onto $S$. If the top vertex of $S$ has a neighbour in $U$, remove such a neighbour, push it onto $S$, and add the corresponding edge to $F$. If the top vertex of $S$ has no neighbour in $U$, then pop it from $S$ and add it to $X$. Stop when $X=V(G)$.

We will use the following straightforward properties of $S$, $U$ and $X$ which hold throughout the process.
\begin{itemize}
  \item The stack $S$ forms an induced path in $G$.
  \item There is no edge of $G$ between $U$ and $X$.
\end{itemize}

\subsection{Sublinear expanders}\label{sec:sublinear-expander}
 For $\ep > 0$ and $t > 0$, let $\rho(x)$ be the function
\begin{equation}
	\rho(x)=\rho(x,\ep,t):=
	\begin{cases}
		0 & \text{if }x<t/5,\\
		\ep/\log^2(15x/t) & \text{if }x\ge t/5,
	\end{cases}\label{eq:expander-factor}
\end{equation}
where, when it is clear from context, we will not write the dependency of $\rho(x)$ on $\ep$ and $t$. Note that when $x\ge t/2$, $\rho(x)$ is decreasing, while $\rho(x)\cdot x$ is increasing.
\begin{defn}[Sublinear expander]\label{def-sublinear-expander}
	A graph $G$ is an \emph{$(\ep,t)$-expander} if for any subset $X\subseteq V(G)$ of size $t/2\le \abs{X}\le |V(G)|/2$, we have $\abs{N_G(X)}\ge \rho(\abs{X})\cdot \abs{X}$.
\end{defn}

Compared with expanders having constant expansion factors, sublinear expanders have a weaker expansion property, but one key advantage of them is that any graph contains a sublinear expander subgraph that, furthermore, is almost as dense as the original graph, as shown by Koml\'os and Szemer\'edi \cite{K-Sz-1,K-Sz-2}. We shall use the following strengthening of their results due to Haslegrave, Kim and Liu \cite{H-K-L}.

\begin{lemma}[\cite{H-K-L}, Lemma 3.2]\label{lem: strong-version-expander-lemma}
	Let $C > 30,  0<\ep  \le 1/(10C), t>0,d>0$ and $\rho(x)=\rho(x,\ep ,t)$ as in \emph{(\ref{eq:expander-factor})}. Then every graph $G$ with $d(G) = d$ has a subgraph $H$ such that $H$ is an $(\ep ,t)$-expander, $d(H)\ge (1-\delta)d$ and $\delta(H)\ge d(H)/2$, where $\delta:=\frac{C\ep }{\log 3}$.
\end{lemma}

The following lemma shows the key property of sublinear expanders that we will utilise.  It roughly says that in a sublinear expander, we can connect two sets $X_1, X_2$ using a short path while avoiding another set $W$ as long as $W$ is a bit smaller than $X_1, X_2$. Although in many applications the bound on the length of such a path will be important, in this paper all we shall actually need is the existence of a path avoiding a certain set.
\begin{lemma}[{Small diameter lemma \cite[Corollary 2.3]{K-Sz-2}}]\label{lem: small-diameter-lemma}
If $G$ is an $n$-vertex $(\ep,t)$-expander, then for any two vertex sets $X_1,X_2$ each of size at least $x\ge t/2$, and a vertex set $W$ of size at most $\rho(x)x/4$, there exists a path in $G-W$ between $X_1$ and $X_2$ of length at most $\frac{2}{\ep}\log^3(\frac{15n}{t})$.
\end{lemma}

\subsection{Isoperimetry}
To find long cycles in subgraphs of hypercubes and Hamming graphs, we will need the following isoperimetric result.

\begin{theorem}[{\cite[Theorem 1]{K-L}}]\label{thm: edge-isoperimetry}
	Every $U \subseteq V(Q^m)$ satisfies $\abs{\partial U}\ge \abs{U}\cdot\log_2(2^m/\abs{U})$.
\end{theorem}
The bound on the order of a subgraph of $Q^m$ with average degree $d$ in (\ref{eq-iso}) then immediately follows.
\begin{prop}\label{prop: edge-isoperimetry}
	Every subgraph $G$ of $Q^m$ with average degree $d$ has at least $2^d$ vertices.
\end{prop}
\begin{proof}
	By Theorem \ref{thm: edge-isoperimetry}, $\abs{\partial V(G)}\ge \abs{G}\cdot\log_2(2^m/\abs{G})$. Since $2\abs{E(G)}+\abs{\partial V(G)}= m\abs G$, we have $\abs{E(G)}=d\cdot\abs{G}/2\le \abs{G}\cdot\log_2\abs{G}/2$. Hence, $\abs{G}\ge 2^d$.
\end{proof}

A similar result for Hamming graphs holds.

\begin{prop}[{\cite[Proposition 2]{Q-T-V}}]\label{prop: discrete-torus-edge-isoperimetry-1}
	Every subgraph $G$ of the Hamming graph $H(m,r)$ has at most $(r-1)\abs{G}\cdot\log_r\abs{G}/2$ edges.
\end{prop}
Consequently, in such a graph $d(G)\leq (r-1)\log_r\abs{G}$, giving the following corollary.
\begin{prop}\label{prop: discrete-torus-edge-isoperimetry-2}
	Every subgraph $G$ of $H(m,r)$ with average degree $d$ has at least $r^{\frac{d}{r-1}}$ vertices.
\end{prop}

\section{Cycles of length linear in crux}\label{sec:sublin}
\subsection{Proof of Theorem~\ref{thm: crux}}
\begin{theorem}[{\cite[Theorem 1]{Kri}}]\label{thm: long-cycle-in-locally-expander}
	Let $k>0, t\ge 2$ be integers. Let $G$ be a graph on more than $k$ vertices, satisfying:
	\[|N_G(W)|\ge t,~\text{for every}~W\subseteq V(G)~\text{with}~k/2\le |W|\le k.\]
	Then $G$ contains a cycle of length at least $t+1$.
\end{theorem}


\begin{proof}[Proof of Theorem~\ref{thm: crux}.]
	Let $\delta=1-\alpha$ and take $C=40,\ep = \frac{\delta}{500}$, so $\delta>\frac{C\ep }{\log 3}$. Write $n_c=c_{\alpha}(G)$ and let $H\subseteq G$ be a subgraph that is an $(\ep , n_c/2)$-expander, guaranteed by Lemma~\ref{lem: strong-version-expander-lemma}. Then $d(H)\ge (1-\delta)d(G)$, by the definition of the crux, we have $n_H:=|H|\ge n_c$. Set $K=\frac{n_H}{n_c}\ge 1$.
	
	As $\rho(x)x$ is increasing in $x$ and $K\ge 1$, by the expansion property of $H$, every set of size $n_H/4\le x\le n_H/2$ has an external neighbourhood of size at least $$\rho\Big(\frac{n_H}{4}\Big)\frac{n_H}{4}=\frac{\ep n_H}{4\log^2(\frac{15n_H/4}{n_c/2})}=\frac{\ep Kn_c}{4\log^{2}(15K/2)}\ge \frac{\ep}{32}\cdot n_c.$$
	We may assume that $\frac{\ep}{32}\cdot n_c\ge 2$, for otherwise we can take a single edge as a degenerate cycle. Then by Theorem~\ref{thm: long-cycle-in-locally-expander}, the graph $H$, hence also $G$, contains a cycle of length at least $\frac{\ep}{32}n_c=\frac{1-\alpha}{16000}c_{\alpha}(G)$.
\end{proof}

\subsection{Proof of Corollary~\ref{ques: cycle-expander}}
A cycle of length $\frac{\ep}{32}c_{\alpha}(G)$ follows from the proof of Theorem~\ref{thm: crux}. The second term $\ep n/(1200\log^2 n)$ follows from the expansion property of sublinear expanders and Theorem~\ref{thm: long-cycle-in-locally-expander}, since any set of size between $n/4$ and $n/2$ has a neighbourhood of size at least $\ep n/(4\log^2(15n/t))$. We give a direct proof for completeness.

First, as $\ep< 1/500$, the conditions on $n$ imply that $n/300\geq t/2$, that $\ep n/(1200\log^2 n)\leq (n/300)\cdot\rho(n/300)/4$, and that $\ep n/(1200\log^2 n)\leq n/300$.

Consequently, if there is a path of length $n/100$, then we are done, because after removing the middle $\ep n/(1200\log^2n)$ vertices of the path, there is still a short path avoiding the middle part connecting the two halves by Lemma \ref{lem: small-diameter-lemma}.
This gives a cycle containing the middle $\ep n/(1200\log^2n)$ vertices of the path. So assume that such a path does not exist.

We run DFS until some point where $\abs{X}=n/3$. Since the stack $S$ always induces a path in $G$, we have $\abs{S}<n/100$, and so $\abs{U}>0.65n$. By Lemma \ref{lem: small-diameter-lemma} and the fact that $S$ is a cut between $X$ and $U$, we have $\abs S>0.3n\cdot\rho(0.3n)/4>\ep n/(1200\log^2n)$. Let $P_1$ be the path induced by $S$ at that point and set $i=2$. Now continue running DFS. Whenever a new vertex is added to $S$, call the new path $P_i$ and increment $i$. Do this until $i=n/3$. By the same reasoning throughout this process we have $\ep n/(1200\log^2n)<\abs{S}<n/100$, and in particular the lower bound implies the first $\ep n/(1200\log^2n)$ vertices of the path never change. Thus we have a set of $n/3$ paths with a long common first section and different endpoints.

Now consider the largest common first section $P$. This corresponds to the point between $P_1$ and $P_{n/3}$ where $S$ is smallest (and equals $P$). Fix $X$ and $U$ corresponding to their values at that point. Again, $P$ is a cut between $X$ and $U$, both of which have size at least $0.32n$. 
Let $P'$ be the subpath of $P$ consisting of the final $\ep n/(1200\log^2n)$ vertices, and $u$ be the same endpoint of $P'$ and $P$. Since $\abs{P}=\abs{S}>\ep n/(1200\log^2n)$, we have $V(P)\setminus V(P')\neq \varnothing$.

Suppose without loss of generality (if not, exchange $X$ and $U$) more than half of the paths $P_1,\ldots,P_{n/3}$ come before this point. This means their endpoints are in $X$; let $Y$ be the set of these endpoints, giving $\abs{Y}\ge 0.16n$.
For any vertex in $Y$, there is a path to $u$ which lies entirely in $X$. Let $Z=U \cup V(P)\setminus V(P')$. Then $Z$ has size more than $0.32n>t/2$. By Lemma \ref{lem: small-diameter-lemma}, there exists a short path in $G-V(P')$ connecting $Y$ and $Z$. Indeed, as there are no edges between $U$ and $X$, the short path connects $Y$ and $V(P)\setminus V(P')$. This gives a cycle containing $P'$ with desired length.

\section{Random subgraphs}\label{sec:cube}

\subsection{Long cycles in random subgraphs of \texorpdfstring{$C_4$}{C4}-free graphs}

We prove Theorem \ref{thm:C4-free} by adapting Riordan's proof \cite{R}.
Recall that $G$ is an $n$-vertex $C_4$-free graph with minimum degree $k$.
Fix $0<\ep <1/10$ and let $C= 10/\ep$. It suffices to show that w.h.p.\ $G_p$ contains a cycle of length at least $(1-20\ep)k^2$ when $pk=\omega(1)$.

Consider a DFS forest $T$ of $G_p$, leaving edges \emph{unrevealed} if they are not needed in the exploration. To be precise, when checking whether the top vertex $v$ of the stack has a neighbour in $U$, we list the remaining edges between $v$ and $U$ (in an arbitrary order) and reveal whether each in turn is in $G_p$ until either we find such an edge or exhaust the list. If an edge $vw$ is found, then we add it to the forest, put $w$ on the top of the stack, and repeat. (While the final forest found is an undirected graph, we also think of edges being associated with an orientation, so that the edge $vw$ just added is oriented from $v$ to $w$; taking these orientations into account makes each component an arborescence.)
If the list is exhausted, we remove the vertex $v$ from the stack and consider the next vertex on top of the stack to repeat.  Note that a vertex is removed from the stack only when no incident edges to $U$ remain (either because they have been revealed or because vertices have been removed from $U$).

We consider each component of the obtained forest $T$ to be rooted at the first vertex to be added to the stack $S$ (that is, the natural root of the associated arborescence), and we consider the set $D(v)$ of \textit{descendants} of a vertex $v$ to be the set of vertices $w$ such that the path from $w$ to the root of its component contains $v$ (note in particular that $v\in D(v)$). Likewise we consider $v$ to be an \textit{ancestor} of $w$ if $w\in D(v)$. For a non-root vertex $v$ of $T$, the neighborhood $N_T(v)$ consists of one ancestor of $v$ called \textit{the parent} of $v$ and possibly some descendants of $v$ called \textit{children} of $v$.

We write $n$ for the order of $G$ and $Q\subseteq G$ for the subgraph consisting of all unrevealed edges. Throughout the process, each edge in $Q$ is present in $G_p$ independently with probability $p$; in particular this means that for any given set of $\ep k$ edges of $Q$, w.h.p.\ at least one is present since $\ep kp\rightarrow \infty$.

We frequently use the following property which results from the use of DFS: every edge of $Q$ joins two vertices in $T$ one of which is an ancestor of the other (and in particular, joins two vertices in the same component of $T$). To see this, let $vw$ be an edge of $Q$, and suppose without loss of generality that $v$ was added to the stack first. If $w$ was added to the stack before $u$ was removed, then $v$ is an ancestor of $w$, since the vertices on the stack always form a path in $T$ (which respects orientations). If not, then $w$ must have remained in $U$ until $v$ was removed from the stack; however, this is impossible since the edge $vw$ was not revealed, and $v$ cannot have left the stack while an unrevealed edge between $v$ and $U$ existed. (See \cite[Lemma 2]{R}.)

Note that we are done provided there is a set $R\subseteq V(T)$ satisfying the following:
\begin{equation}\sum_{v\in R}\bigl|\{u:uv\in Q, (1-20\ep)k^2\leq d_T(u,v)<\infty\}\bigr|\geq \ep k,\label{one}\end{equation}
where $d_T(u,v)$ is the distance in $T$, since then w.h.p.\ at least one of these $\ep k$ edges is present, say $uv$, and creates a cycle of length at least $(1-20\ep)k^2$ together with the path in $T$ from $u$ to $v$. Thus we assume from now on that \eqref{one} is not true for any set $R$.

The property described above means that $uv\in Q$ with $u\in V(T)$ already implies $d_T(u,v)<\infty$, and that the distance requirement in \eqref{one} only rules out some descendants and ancestors of $u$ that are too close. Note also that every ancestor of $u$ has a different distance to $u$.

A vertex is \textit{full} if it has at least $(1-\ep)k$ incident edges in $Q$, meaning that most of the edges incident with $v$ were never explored.
As the forest $T$ has at most $n-1$ edges, standard concentration inequalities show that w.h.p.\ at most $2n/p=o(kn)$ edges are revealed in the whole process; and so w.h.p.\ all but $o(n)$ vertices are {full}. We may therefore assume in what follows that all but $o(n)$ vertices are {full}.

\begin{claim}\label{cl: k2 expansion}
	For any set $A$ of $Ck$ full vertices, we have $|N_Q(A)|\geq (1- 4\ep)k^2$.
\end{claim}
\begin{proof}
	Consider the bipartite graph $H=Q[A,B]$ consisting of the unrevealed edges between $A$ and $B$ where $B=N_Q(A)$.
	Note that $G[A]$ is a $C_4$-free graph with $Ck$ vertices, hence by standard bounds on $\ex(Ck,C_4)$, e.g.~\cite{Rei}, it contains at most $(Ck)^{1.5} < \ep^2 k^2$ edges for $k$ sufficiently large.
	Then, as the vertices in $A$ are full, $H$ contains at least $(1-\ep-\ep^2)C k^2$ edges.
	
	If $\sum_{v\in B} \binom{d_H(v)}{2}>\binom{|A|}{2} = \binom{Ck}{2}$, then there exists a pair of vertices in $A$ having two common neighbours, a contradiction to the $C_4$-freeness of $G$. Hence, by convexity of the function $f(x)=~\binom{x}{2}$, we have
	\[\binom{Ck}{2} \geq \sum_{v\in B} \binom{d_H(v)}{2} \geq  |B| \binom{(1-\ep-\ep^2)Ck^2/|B|}{2} \geq (1-3\ep)\Big(\frac{C^2k^4}{2|B|} - \frac{Ck^2}{2} \Big).\]
	As $C>10/\ep$, this yields that $|B|\geq (1-3\ep)(1- \frac{1}{C+1})k^2 \geq (1-4\ep)k^2$.
\end{proof}

We say that a vertex is \textit{poor} if it has at most $\ep k^2$ descendants, and \emph{rich} otherwise. We wish to show that at most $o(n)$ vertices are poor. In \cite{R} where we aim for a cycle of length $(1-o(1))k$, and the definition of poor and the condition \eqref{one} are adjusted appropriately by replacing $k^2$ with $k$, this is immediate, since if $v$ is both poor and full then $\{v\}$ satisfies the equivalent of \eqref{one} (at most $\ep k$ incident edges are not in $Q$, at most $\ep k$ go to descendants, and so the remainder go to ancestors, of which at most $20\ep k$ are too close). However, this does not translate to our setting. Consequently establishing that there are few poor vertices is the main difficulty in extending the proof.

\begin{lemma}\label{lem:no-useful-set}If \eqref{one} does not hold for any set $R$, then $o(n)$ vertices are poor.\end{lemma}
\begin{proof}
	Let $W$ be a subset of children of some vertex $v$ and write $R(W)=\bigcup_{w\in W}D(w)$. Suppose $2Ck\leq\abs{R(W)}\leq\ep k^2$. If some set $S$ of at least $Ck$ vertices in $R(W)$ are full, then by Claim~\ref{cl: k2 expansion}, we may choose $(1-4\ep)k^2$ neighbours of vertices in $S$ via edges of $Q$. Recall that each edge in $Q$ goes to a descendant or ancestor, so each of these neighbours is either in $R(W)$ or is an ancestor of $v$. However, at least $(1-5\ep)k^2$ of these neighbours are not in $R(W)$ and must be ancestors of $v$; since $v$ has at most one ancestor at each distance, at least $\ep k^2$ of them are at distance at least $(1-6\ep)k^2$ from $R(W)$, and so \eqref{one} holds for $R(W)$. Thus, for a vertex $v\in V(T)$ and a subset $W$ of children of $v$ satisfying $2Ck\leq\abs{R(W)}\leq\ep k^2$, at most half of the vertices in $R(W)$ are full.

	Write $\mathcal P$ for the set of poor vertices, and $\mathcal F$ for the set of full vertices. We divide $\mathcal P$ into groups, according to their nearest rich ancestor. However, there may be some poor vertices with no rich ancestor, corresponding to small components of $T$; we deal with these separately. Write $\mathcal P_v$ for the set of poor vertices whose nearest rich ancestor is $v$. Notice that $\mathcal P_v=R(W_v^{\mathrm{poor}})$, where $W_v^{\mathrm{poor}}$ is the set of poor children of $v$. Write $A$ for the set of vertices $v$ with $\mathcal P_v\neq\varnothing$. Finally, write $\mathcal P^*$ for the set of poor vertices with no rich ancestor.
	
	First, note that $\mathcal P^*$ consists of all vertices in components of $T$ of order at most $\ep k^2$. Let $X$ be the vertices of some component of $T$ having order $\ell\leq\ep k^2$. Since $G[X]$ is $C_4$-free, it contains at most $\ell^{1.5}\leq\ep^{0.5}k\ell$ edges. Suppose $X$ contains at least $3\ell/4$ full vertices. Then, since any edges of $Q$ meeting $X$ are in $G[X]$, $G[X]$ has at least $3(1-\ep)k\ell/8$ edges, a contradiction since $\ep<1/10$. Consequently at least one quarter of the vertices in any such component, and hence of $\mathcal P^*$, are not full. Since there are $o(n)$ such vertices, $\abs{\mathcal P^*}=o(n)$.
	
	We now split $A$ into two parts, which we deal with in different ways. Set
	\[A_1=\{v\in A:\abs{\mathcal P_v\cap\mathcal F}\leq 3\abs{\mathcal P_v}/4\} \text{ and } A_2 = A\setminus A_1.\]
	Recall that we may assume all but at most $o(n)$ vertices are full. Since at least one quarter of vertices in $\bigcup_{v\in A_1}\mathcal P_v$ are not full, it follows that $\bigl|\bigcup_{v\in A_1}\mathcal P_v\bigr|=o(n)$. Thus it suffices to show that
	$\bigl|\bigcup_{v\in A_2}\mathcal P_v\bigr|=o(n)$.
	
	Suppose $v\in A$ satisfies $\abs{\mathcal P_v}\geq 2Ck$. Then we may divide $W_v^{\mathrm{poor}}$ into disjoint subsets $W_1,\ldots,W_r,L$ such that each of $R(W_1),\ldots,R(W_r)$ have size between $2Ck$ and $\ep k^2$ and $R(L)$ has size less than $2Ck$, for some $r\geq 1$. It follows that at most half of the vertices in $R(W_i)$ are full for each $i$, and since $r\geq 1$ and $\abs{R(L)}<\abs{R(W_1)}$, at most three quarters of the vertices in $\mathcal P_v$ are full. Thus $v\in A_1$. In particular, this is the case for any vertex $v$ which is rich but has no rich children.
	
	In order to show that $\bigl|\bigcup_{v\in A_2}\mathcal P_v\bigr|=o(n)$, we will associate each $y\in\bigcup_{v\in A_2}\mathcal P_v$ with a set $Z_y$ of size $\floor{\ep k/(4C)}=\omega(1)$, ensuring that all of these sets are disjoint. Since the total size of all sets $Z_y$ is at most $n$, it will follow that $\bigl|\bigcup_{v\in A_2}\mathcal P_v\bigr|\leq n/\floor{\ep k/(4C)}=o(n)$.
	
	We will construct the sets $Z_y$ in several stages. We let $Y_0=A_2$ and in each stage $i$ we will choose a subset $X_i\subseteq Y_{i-1}$, and construct $Z_y$ for each $y\in\bigcup_{v\in X_i}\mathcal P_v$. Setting $Y_i=A_2\setminus(\bigcup_{j<i}X_j)$ to be the remaining vertices in $A_2$ after $i-1$ stages, we continue until $Y_i=\varnothing$.
	
	In stage $i$, choose $v_i\in Y_i$ as close to the root of its component as possible, so that $u\not\in Y_i$ for each ancestor $u$ of $v_i$. Define a path $P_i$, starting at $v_i$ and proceeding downwards, using only rich vertices, until one of the following is satisfied:
	\begin{enumerate}
		\item\label{enough} The total size of $\bigcup_{w\in P_i\cap Y_i}\mathcal P_w$ is at least $2Ck$, or
		\item\label{stuck} the last vertex on the path has no rich children.
	\end{enumerate}
	Clearly it is possible to construct such a path, since so long as neither 1 nor 2 is satisfied we can extend the path by adding a rich child of the last vertex. Write $x_i$ for the last vertex of $P_i$. We then choose $X_i$ to be the set $P_i\cap Y_i$.
	
	Suppose \ref{enough} is satisfied. In this case, the last vertex added to the path must be in $Y_i\subseteq A_2$. Since every vertex $w\in A_2$ satisfies $\abs{\mathcal P_w}\leq 2Ck$, we must have $2Ck\leq |\bigcup_{w\in X_i}\mathcal P_w| \leq 4Ck$. Furthermore, since $X_i\subseteq Y_i\subseteq A_2$, at least three quarters of the vertices in $\bigcup_{w\in X_i}\mathcal P_w$ are full. Consequently, Claim~\ref{cl: k2 expansion} ensures that there are at least $(1-4\ep)k^2$ distinct vertices adjacent to $\bigcup_{w\in X_i}\mathcal P_w$ by unrevealed edges. Since every unrevealed edge from a vertex goes to an ancestor or descendant, all such vertices must be either in $\bigcup_{w\in X_i}\mathcal P_w$, or on $P_i$, or ancestors of $v_i$. If $\abs{P_i}\leq \ep k^2$ then at least $(1-5\ep)k^2-4Ck\leq(1-6\ep)k^2$ of the vertices must be ancestors of $v_i$ (for $k$ sufficiently large). Of these, at least $14\ep k^2$ must be at least at distance $(1-20\ep)k^2$ from $v_i$ (since it has at most one ancestor at each distance), and so $\bigcup_{w\in X_i}\mathcal P_w$ satisfies \eqref{one}, a contradiction. Thus $\abs{P_i}\geq \ep k^2$. We may therefore choose disjoint sets $Z_y\subseteq P_i$ of size $\floor{\ep k/(4C)}$ for each $y\in \bigcup_{w\in X_i}\mathcal P_w$.
	
	Alternatively, suppose \ref{enough} is not satisfied, and so \ref{stuck} is satisfied and  $\bigl|\bigcup_{w\in X_i}\mathcal P_w\bigr|<2Ck$. Note that, since $x_i$ is rich but has only poor children, it is in $A_1$ and so not in $Y_i$. Also we have $|D(x_i)|\geq \ep k^2$. We may therefore choose disjoint sets $Z_y\subseteq D(x_i)$ of size $\floor{\ep k/(4C)}$ for each $y\in \bigcup_{w\in X_i}\mathcal P_w$.
	
	We now proceed to stage $i+1$, and continue in this manner until we reach some stage $j$ with $Y_j=\varnothing$; since $Y_i$ decreases at each stage, this eventually happens. It only remains to show that the sets $Z_y$ chosen in different stages are disjoint. Each such set constructed in stage $i$ is either chosen from $P_i$, in which case it consists only of rich vertices, or from $D(x_i)$, in which case it consists only of poor vertices. It suffices to show that the paths $P_i$ are disjoint, since then the rich sets chosen in different stages come from disjoint paths, and the poor sets chosen in different stages have different nearest rich ancestors. Suppose this is not the case, so that $w\in P_i\cap P_j$ for some $i<j$. Then, since both $v_i$ and $v_j$ are ancestors of $w$, we must have that either $v_i$ is an ancestor of $v_j$ or vice versa. Also, we have $v_i\in Y_i$ and $v_j\in Y_j\subset Y_i$. As we have chosen $v_i\in Y_i$ as close to the root of its component as possible, we know that $v_j$ cannot be an ancestor of $v_i$. However, as $w\in P_i$, if $v_i$ is an ancestor of $v_j$ it follows that $v_j\in P_i$, and hence $v_j\in X_i$, a contradiction since $v_j\in Y_j$.
	
	This completes the proof that the sets $Z_y$ for $y\in \bigcup_{v\in A_2}\mathcal P_v$ are disjoint. Since each has size $\omega(1)$, it follows that $\bigl|\bigcup_{v\in A_2}\mathcal P_v\bigr|=o(n)$. Thus $\abs{\mathcal P}=o(n)$, as required.\end{proof}

A path in a rooted tree is \emph{vertical} if one of its endpoints is a descendant of the other.
Define a vertex $v\in V(T)$ to be \textit{light} if $\bigl|D_{\leq(1-10\ep)k^2}(v)\bigr|\leq (1-9\ep)k^2$, where $D_{\le i}(v)\subseteq D(v)$ are the descendants within distance $i$ of $v$. If a vertex $v\in V(T)$ is not light, we call it \textit{heavy}.
Let $\mathcal H$ be the set of heavy vertices. The proof of the following lemma, which we include for completeness, is the same as \cite[Lemmas 5, 6]{R} up to slight changes in the parameters.
\begin{lemma}\label{lem:long-path}Suppose that $T$ contains $o(n)$ poor vertices and $Y\subseteq V(T)$ satisfies $\abs{Y}=o(n)$. Then for $k$ sufficiently large, $T$ contains a vertical path $P$ of length $2C k^2$, containing at most $\ep^2 k^2$ vertices in $Y\cup \mathcal H$.\end{lemma}
\begin{proof}Define the \textit{height} of a vertex to be the maximum distance to a descendant. We first show that almost all vertices are at height at least $\ep^{-2}k^2$.

For each rich vertex $v$ of height less than $\ep^{-2}k^2$, let $S(v)$ be a set of $\lceil \ep k^2\rceil$ descendants of $v$ with total distance from $v$ as large as possible. Notice that this implies each $w\in S(v)$ has at most $\lceil \ep k^2\rceil-1$ descendants, so is poor. We count pairs $(v,w)$ with $w\in S(v)$; since each such pair has $w$ being one of the $o(n)$ poor vertices, and $v$ being one of the $\ep^{-2}k^2$ lowest ancestors of $w$, there are at most $\ep^{-2} k^2o(n)$ pairs. However, each rich vertex of height less than $\ep^{-2}k^2$ is in at least $\ep k^2$ pairs, so there are at most $\ep^{-3}o(n)+o(n)=o(n)$ vertices of height less than $\ep^{-2}k^2$.

We next show that there are few heavy vertices. We count pairs $(u,v)$ of distinct vertices where $u$ is an ancestor of $v$ at distance at most $(1-10\ep)k^2$. Since each vertex has at most one ancestor at each distance, there are at most $(1-10\ep)k^2n$ pairs. Since all but $o(n)$ vertices are of height at least $\ep^{-2}k^2>(1-10\ep)k^2$, and so are the first vertex in at least $(1-10\ep)k^2$ pairs, and since each heavy vertex is the first vertex in at least $(1-9\ep)k^2$ pairs, we have
$$(1-10\ep)k^2 n \geq (1-10 \ep)k^2(n-o(n)-|\mathcal{H}|)+ (1-9\ep)k^2 |\mathcal{H}|,$$
implying $(1-10\ep)k^2o(n)\geq \ep k^2\abs{\mathcal H}$, and so $\abs{\mathcal H}=o(n)$.

Finally, consider the pairs $(u,v)$ where $v\in Y\cup \mathcal H$ and $v\in D_{\le \ep^{-2}k^2}(u)$. Since each vertex $v\in Y\cup \mathcal H$ is the second vertex in at most $\ep^{-2}k^2$ pairs, and $\abs{Y\cup\mathcal H}=o(n)$, there are $o(k^2n)$ pairs. Therefore at most $o(n)$ vertices $u$ appear in more than $\ep^2k^2$ pairs as a first entry, and as shown above at most $o(n)$ vertices have height less than $\ep^{-2}k^2$. Choosing a vertex $u$ in neither of these categories, there exists a vertical path of length $\lceil\ep^{-2}k^2\rceil$ with top vertex $u$, and any such path contains at most $\ep^2k^2$ vertices in $Y\cup\mathcal H$, as required.
\end{proof}
We are now ready to complete the proof of Theorem \ref{thm:C4-free}. We are done if any set satisfies \eqref{one}, so assume not. Then Lemmas \ref{lem:no-useful-set} and \ref{lem:long-path} ensure the long vertical path $P$ described above exists.
Write $Z$ for the set of vertices on $P$ which are both full and light. We order $Z$ according to height on the path, and will consider blocks of $Ck$ consecutive vertices of $Z$ in this ordering. By Lemma \ref{lem:long-path}, there are at most $\ep^2k^2$ vertices on the path which are not in $Z$, so the total distance on the path between the top and bottom vertices of any such block is at most $\ep^2k^2+Ck<\ep k^2$. By Claim~\ref{cl: k2 expansion}, any block $A$ satisfies $|N_Q(A)|\geq (1-4\ep)k^2$.

Fix some block $A$, and let $u$ and $v$ be the highest and lowest vertices of that block respectively. Recall that every vertex in $N_Q(A)$ is either an ancestor or a descendant of its neighbour in $A$, and hence either an ancestor or a descendant of $u$. Since $u$ is light, it has at most $(1-9\ep)k^2$ descendants within distance $(1-10\ep)k^2$, hence $|N_Q(A)\cap D_{\leq (1-10\ep)k^2}(u)|\leq (1-9\ep)k^2$.

We also have $|N_Q(A)\cap (D(v)\setminus D_{\leq (1-10\ep)k^2}(u))|\leq \ep k$. Indeed, if not, let $R$ be a set of at least $\ep k$ such vertices. Then the vertices in $R$ are at distance at least $(1-10\ep)k^2$ from $u$ and every vertex in $A$ is within distance $\ep k^2$ of $u$, so every edge $uv$ of $Q$ between $R$ and $A$ satisfy $d_T(uv) \geq (1-11\ep)k^2$, a contradiction to \eqref{one}.
Hence, we have $|N_Q(A)\cap D(u)| \leq (1-9\ep)k^2 + \ep k$.
As $|N_Q(A)|\geq (1-5\ep)k^2$, at least $4\ep k^2-\ep k$ neighbours of vertices in $A$ are ancestors of its highest vertex $u$.

Taking $V_0$ to be the bottom $Ck$ vertices of $Z$ we know that, for $k$ sufficiently large, these have at least $4\ep k^2-2\ep k>\ep k$ neighbours at least distance $4\ep k^2-2\ep k\geq  3\ep k^2$ above the highest vertex of $V_0$, so w.h.p.\ we can find a $v_0\in V_0$ and $u_0$ at least this distance above, connected by an edge of $Q$ which is present in $G_p$. Then we choose $V_1$ to be the highest $Ck$ vertices in $Z$ below $u_0$ and continue. Note that those $Ck$ vertices are disjoint from $V_0$ as $3\ep k^2 > Ck + \ep^2 k^2$.
Note that we go up at least $3\ep k^2$ steps from the top vertex of $V_0$ to $u_0$ and down at most $\ep^2 k^2$  steps from $u_0$ to the top  of $V_1$. Since $0<\ep<1/10$ and $d_T(v_0,u_0)<k^2$ (for otherwise we have a length-$k^2$ cycle), and the path $P$ has length $2C k^2$, w.h.p.\ we may continue in this way to find overlapping `chords' $v_iu_i$ for $0\le i \le C$. Since $d_T(u_i,v_{i+2})\ge 3\ep k^2-2\ep^2 k^2-Ck>\ep k^2$, w.h.p.\ there is a cycle of length at least $C\ep k^2 \geq k^2$ consisting of these chords together with the sections of the path $v_0\cdots v_1$ and $u_i\cdots v_{i+2}$ for $0\le i \le C-2$, and $u_{C-1}\cdots u_{C}$. See Figure \ref{fig:cycle-from-path} for an illustration.
\begin{figure}[ht]
\centering
\begin{tikzpicture}
\draw[->] (0,0) -- (10.2,0) node[anchor=west]{$P$};
\draw [decoration={brace, mirror}, decorate] (0,-.1) -- node [anchor=north]{$V_0$} (1,-.1);
\filldraw (0.3,0) circle (0.05) node [anchor=south east] {$v_0$};
\filldraw (2.55,0) circle (0.05) node [anchor=south west] {$u_0$};
\draw[very thick] (0.3,0) to[out=60, in=120] (2.55,0);
\draw [decoration={brace, mirror}, decorate] (1.5,-.1) -- node [anchor=north]{$V_1$} (2.5,-.1);
\draw[very thick] (0.3,0) -- (2,0);
\filldraw (2,0) circle (0.05) node [anchor=south east] {$v_1$};
\filldraw (4.05,0) circle (0.05) node [anchor=south west] {$u_1$};
\draw[very thick] (2,0) to[out=60, in=120] (4.05,0);
\draw [decoration={brace, mirror}, decorate] (3,-.1) -- node [anchor=north]{$V_2$} (4,-.1);
\draw[very thick] (2.55,0) -- (3.8,0);
\filldraw (3.8,0) circle (0.05) node [anchor=south east] {$v_2$};
\filldraw (6.05,0) circle (0.05) node [anchor=south west] {$u_2$};
\draw[very thick] (3.8,0) to[out=60, in=120] (6.05,0);
\draw [decoration={brace, mirror}, decorate] (5,-.1) -- node [anchor=north]{$V_3$} (6,-.1);
\draw[very thick] (4.05,0) -- (5.2,0);
\filldraw (5.2,0) circle (0.05) node [anchor=south east] {$v_3$};
\filldraw (7.75,0) circle (0.05) node [anchor=south west] {$u_3$};
\draw[very thick] (5.2,0) to[out=60, in=120] (7.75,0);
\draw [decoration={brace, mirror}, decorate] (6.7,-.1) -- node [anchor=north]{$V_4$} (7.7,-.1);
\draw[very thick] (6.05,0) -- (7.2,0);
\filldraw (7.2,0) circle (0.05) node [anchor=south east] {$v_4$};
\filldraw (9.55,0) circle (0.05) node [anchor=south west] {$u_4$};
\draw[very thick] (7.2,0) to[out=60, in=120] (9.55,0);
\draw[very thick] (7.75,0) -- (9.55,0);
\end{tikzpicture}
\caption{An example cycle (shown in bold) constructed from the vertical path $P$, drawn horizontally for ease of presentation (higher vertices are positioned further to the right).}\label{fig:cycle-from-path}
\end{figure}
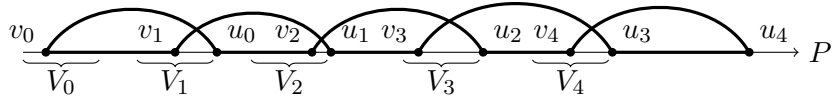

\subsection{Long cycles in random subgraphs of hypercubes}

To prove Theorem \ref{thm: long-cycle-in-hypercube2}, we use concentration of the size of the giant component to show that w.h.p.\ there is no small separators. This idea is not new and appeared earlier in the work of Krivelevich, Lubetzky and Sudakov~\cite{KrLS}. To carry out this argument, we need a result relating separability of graphs to separator size; we first give the necessary definitions.

\begin{defn}
Given a graph $G=(V, E)$ on $n$ vertices, a vertex set $S\subseteq V$ is called a \emph{separator} if there is a partition $V=A\cup B\cup S$ of the vertex set of $G$ such that $G$ has no edges between $A$ and $B$, and $|A|, |B|\le 2n/3$. 	
\end{defn}

\begin{defn}
Let $s,t$ be positive integers. A graph $G$ is \emph{$(s,t)$-separable} if there exists a vertex subset $S\subseteq V(G)$ such that $|S|\le s$ and every component of $G-S$ has at most $t$ vertices.
\end{defn}

\begin{lemma}\label{claim: separator}
Let $G$ be a graph with $n$ vertices and fix $t,r>0$. If $G$ is not $(\frac{4 n^2}{r t}, t)$-separable, then $G$ has a subgraph $H$ such that $\abs{H}\ge t$ and $H$ has no separator with size at most $\frac{1}{r}\abs{H}$.
\end{lemma}
\begin{proof}
Suppose that every subgraph $H$ of $G$ with at least $t$ vertices has a separator with size at most $\frac{1}{r}\abs{H}$. Then $G$ has a separator $S$ such that $|S|\le \frac{1}{r}\abs{G}$ and $V(G)\setminus S=X_1 \dot{\cup} X_2$ with $|X_1|,|X_2|\le \frac{2n}{3}$ and $e_G(X_1,X_2)=0$. For each $X_i$ ($i\in \{1,2\}$), if $|X_i|\geq t$, then $G[X_i]$ has a separator $S_i$ such that $|S_i|\le \frac{1}{r}|X_i|$ and $X_i\setminus  S_i=X_{i1} \dot{\cup} X_{i2}$ with $|X_{i1}|,|X_{i2}|\le \frac{2|X_i|}{3}\le (\frac{2}{3})^2n$ and $e_G(X_{i1},X_{i2})=0$. For each $X_{ij}$ ($i,j\in \{1,2\}$), if $|X_{ij}|\geq t$, then $G[X_{ij}]$ has a separator $S_{ij}$ such that $|S_{ij}|\le \frac{1}{r}|X_{ij}|$ and $X_{ij}\setminus S_{ij}=X_{ij1} \dot{\cup} X_{ij2}$ with $|X_{ij1}|,|X_{ij2}|\le \frac{2|X_{ij}|}{3}\le (\frac{2}{3})^3n$ and $e_G(X_{ij1},X_{ij2})=0$. We repeat this to obtain $S_{ijk}, X_{ijk1}, X_{ijk2}$ ($i,j,k\in \{1,2\}$) and so on. Assume that this process stops when $S_{i_1i_2i_3\ldots i_{\ell}}, X_{i_1i_2i_3\ldots i_{\ell+1}}$ are obtained, \ie each $X_{i_1i_2i_3\ldots i_{\ell+1}}$ has size less than $t$.
For each $k\leq \ell+1$ let $\mathcal{A}^k = \{ i_1\dots i_k: X_{i_1\dots i_k} \text{ is defined} \}.$

As $t\le |X_{i_1i_2i_3\ldots i_{\ell}}|\le (\frac{2}{3})^{\ell} n$, we know that $\ell \le \log_{3/2} (n/t)$. Let $S^0=S$ and for $1\le k \le \ell$, $S^{k}=\bigcup\limits_{ i_1\dots i_k \in \mathcal{A}^k}S_{i_1i_2i_3\ldots i_{k}}$. Then
\begin{align*}
|S^{k}|&\le
\sum\limits_{i_1\dots i_k \in \mathcal{A}^k}\frac{1}{r}|X_{i_1i_2i_3\ldots i_{k}}| \le 2^{k}\cdot \frac{1}{r}\cdot \Big{(}\frac{2}{3}\Big{)}^{k} n \le \Big{(}\frac{4}{3}\Big{)}^{k}\cdot \frac{ n}{r}.
\end{align*}
Let $S^{*}=\bigcup_{0\le k \le \ell}S^{k}$. Then $|S^{*}|\le 3\cdot(\frac{4}{3})^{\ell+1}\cdot\frac{ n}{r}\le \frac{4 n^2}{r t}$ and every component in $G-S^{*}$ has size less than $t$. Hence, $G$ is $(\frac{4 n^2}{r t}, t)$-separable, a contradiction.
\end{proof}

By taking $r= 4\psi(n)^3$ and $t=\frac{n}{\psi(n)}$, where $\psi(n)=n^{o(1)}$, we have the following corollary.
\begin{cor}\label{cor: separator}
If $G$ is not $(\frac{n}{\psi(n)^2},\frac{n}{\psi(n)})$-separable, then $G$ has a subgraph $H$ such that $\abs{H}\ge \frac{n}{\psi(n)}$ and $H$ has no separator with size at most $\frac{1}{4\psi(n)^3}\abs{H}$.
\end{cor}

Write $\mathcal{C}_1(G)$ for the largest component in a graph $G$. Let $\ep >0$ be fixed and sufficiently small. Set $p=(1+\ep)/m$ and $p'=(1-{\frac{\ep}{4}})p>(1+{\frac{\ep}{2}})/m$. Write $p_1=(1+{\frac{\ep}{4}})/m$ and choose $p_2\geq {\frac{\ep}{4m}}$ such that $(1-p_1)(1-p_2)=1-p'$ and $n=2^m$. We assume that $m$ is sufficiently large.
For our argument, we prove the following claim. The same result was proved by Ajtai, Koml\'os and Szemer\'edi in \cite{AKS} with a weaker bound of $1-o(1)$ on the probability. However, the bound on the probability that their argument provide is much worse than the following near-exponential bound on the probability, which is crucial for our purpose.
\begin{claim}\label{claim: largest component}
There exists $c=c(\ep)$ satisfying the following: $\mathbb{P} [|\mathcal{C}_1(Q^m_{p'})|\geq c n]\ge 1-\exp(-n/m^{14})$.
\end{claim}
\begin{proof}
We prove this in two steps. The first step (clustering) is performed in $Q^m_{p_1}$, and we deduce that w.h.p.\ $\Omega(2^m)$ vertices are contained in components of size at least $m^4$ and most of vertices are adjacent to at least one such a component. For the second step (sprinkling), we mainly follow the sprinkling process in \cite[Section 1.3]{JEMS}: add the edges of $Q^m_{p_2}$ and show that they can connect many of the clusters of size at least $\sqrt{m}$ into a giant cluster of size $\Theta(2^m)$.

\medskip\noindent\textbf{Step 1.}
Let $V=V(Q^m)$.
Let the random variable $B=B(Q^m)$ be the set of  vertices in $Q^m_{p_1}$ that belong to a component of order at least $m^4$.
By the main theorem in \cite{AKS}, there exists $c_1=c_1(\ep/12)<1/12$ such that for any $q\geq  (1+\ep/12)/m$,
\begin{align}\label{eq: AKS result}
\mathbb{P}[\mathcal{C}_1(Q^m_{q})> 12 c_1 2^m ]\ge 1- c_1.\end{align}
Since $c_1 2^m > m^4$, it follows that $\mathbb{E}[|B|]\geq 6 c_1 2^m$.

For a vertex $v\in V(Q^m)$, we can find distinct vertices $v_1,\dots, v_{\ep m/12} \in N_{Q^m}(v)$ and vertex-disjoint subhypercubes $Q_1,\dots, Q_{\ep m/12}$ of dimension $(1-\ep /12)m$ in $Q^m$ with $v_i\in Q_i$ for each $i$.

Note that conditioning on the existence of a component of size $12c_1 |Q^m|$ in $Q^m_q$, the probability that such a component contains a specific vertex $v$ is at least $12c_1$ as $Q^m$ is vertex-transitive. Hence, the equation \eqref{eq: AKS result} (with $(1-\ep /12)m$ playing the role of $m$) implies that
the vertex $v_i$ belongs to a component of size $12c_1 |Q_i| \geq m^4$ in $(Q_i)_{p_1}$ with probability at least $ 12c_1 (1-c_1) \geq c_1$.
As $Q_1,\dots, Q_{\ep m/12}$ are disjoint subgraphs of $Q^m$, those events are mutually independent.
Moreover, if one such event happens, then we have $v\in N_{Q^m}[B]$, where we write $N_{Q^m}[B] = B\cup N_{Q^m}(B)$.
Hence, we have
$$\mathbb{E}[|V\setminus N_{Q^m}[B]|] =
\sum_{v\in V} \mathbb{P}[ v \notin N_{Q^m}[B] ] \leq (1-c_1)^{\ep m/12} \cdot 2^m \leq \frac{2^{m}}{m^2}.$$

Enumerate edges of $Q^m$ as $e_1,e_2,\ldots,e_{m2^{m-1}}$; let $I_i$ be the indicator random variable that $e_i\in E(Q^m_{p_1})$ and let $\mathcal{F}_i$ be the $\sigma$-algebra generated by $(I_j)_{j\leq i}$. Consider the edge-exposure martingale $X_0,X_1,\ldots,X_n$ and $Y_1,\dots, Y_n$ with
$$X_i=\mathbb{E}[|B| : \mathcal{F}_i] \text{ and } Y_i = \mathbb{E}[ |V\setminus N_{Q^m}(B)| : \mathcal{F}_i].$$
Note that changing one $I_i$ changes $|B|$ by at most $2 m^4$ and $|N_Q[B]|$ by at most $2m^{5}$, since any vertex for which $e_i$ is critical is in a component of order less than $m^4$ in $Q^m_{p_1}-e_i$ containing exactly one endpoint of $e_i$, and such a component has at most $m^5$ neighbours in $Q^m$. Thus the martingales are $2m^4$-Lipschitz and $2m^{5}$-Lipschitz respectively, and by Azuma's inequality we have
\begin{align*}\mathbb{P}\left[|B|< 3 c_1 2^m\right]&\leq\mathbb{P}\left[|B|<\mathbb{E}[|B|] - 3 c_1 2^m\right] \leq\exp\bigg(-\frac{ 9 (c_1)^2 2^{2m}}{ 2(2m^4)^2\cdot m 2^{m-1}}\bigg)
 \leq \exp\Big(- \frac{2^m}{m^{10}}\Big),\end{align*}
  \begin{align*}
\mathbb{P}\left[ |V\setminus N_{Q^m}[B]| >  \frac{2^{m+1}}{m} \right] &\leq\mathbb{P}\left[|V\setminus N_{Q^m}[B]| >\mathbb{E}\left[|V\setminus N_{Q^m}[B]|\right] +  \frac{2^m}{m}\right] \\& \leq\exp\bigg(-\frac{  2^{2m}/m^2}{ 2(2 m^{5})^2 \cdot m2^{m-1}}\bigg)
\leq \exp\Big(- \frac{2^m}{4m^{13}}\Big).
\end{align*}

\medskip\noindent\textbf{Step 2.}  From Step 1, we have $|B|\geq 3c_1 2^m$ and $|V\setminus N_{Q^m}[B]| \leq 2^{m+1}/m$
with probability at least $1-2\exp(-2^m/4m^{13})$. We say that \emph{sprinkling fails} when these high probability events happen but $|\mathcal{C}_1(Q^m_{p_1}\cup Q^m_{p_2})|\leq  c_1 2^m$. If sprinkling fails, then we can partition $B =C\dot{\cup} D$ such that $|C|,|D|\geq  c_1 2^m$, each of $C$ and $D$ is a union of components in $Q^m_{p_1}$, and any $C$-$D$ path in $Q^m$ has an edge missing in $Q^m_{p_2}$. Since every component of $Q^m_{p_1}$ meeting $B$ has size at least $m^4$, the number of partitions meeting the second condition is at most $2^{{2^m}/m^4}$.

It follows from Harper's vertex isoperimetric inequality for the hypercube \cite{Harper} that any set $X\subset V(Q^m)$ of size at most $2^{m-1}$ satisfies $|N_{Q^m}(X)|\geq(1+o(1))|X|\sqrt{2/(\pi m)}$. Consequently, for a particular partition $C\dot{\cup} D$ with $|C|,|D|\geq c_1 2^m$ there is no $C$-$D$ separating set of size less than $\frac{c_1}{100 \sqrt{m}}\cdot 2^m$, so by Menger's theorem there exist at least this many internally vertex-disjoint $C$-$D$ paths in $Q^m$.

Take  such a collection $\mathcal{P}$ of paths with the minimum total sum of lengths.  Note that a path in $\mathcal{P}$ has at most four vertices in $N_{Q^m}[B]$.
 Indeed, if a vertex $u_i$ in the path $u_1u_2\dots u_s$ with $u_1 \in C, u_s\in D$ and $3\leq i\leq s-2$ has a neighbour $w$ in $B = C\cup D$,
then either the path $u_1\dots u_i w$ or the path $w u_i u_{i+1}\dots u_s$ can replace the path $u_1\dots u_s$ in $\mathcal{P}$ to contradict the minimality of $\mathcal{P}$.
Hence, at most $|V(Q^m)\setminus N_{Q^m}[B]| \leq 2^{m+1}/m$ paths in $\mathcal{P}$ have length at least $4$ and at least $\frac{c_1}{100 \sqrt{m}}\cdot 2^m - \frac{2^{m+1}}{m} \geq \frac{c_1}{200 \sqrt{m}}2^m$ paths have length at most $3$. Hence, the probability that all such paths have an edge missing in $Q_{p_2}^m$ is at most
\begin{align*}\Big(1-\big(\frac{\ep}{4m}\big)^{3}\Big)^{\frac{c_1 2^m}{200\sqrt{m}}}&<\exp\Big(-\frac{1}{2}\big(\frac{\ep}{4m}\big)^{3}\cdot \frac{c_1 2^m}{200\sqrt{m}}\Big) <{2^{- 2^{m+2}/m^{4}}}.
\end{align*}
Consequently the probability that sprinkling fails is at most
\[ 2^{2^m/m^4} \cdot 2^{-2^{m+2}/m^4} \leq \exp(-2^{m}/m^4).\]
By the above two steps, we obtain that
\[\mathbb{P} [|\mathcal{C}_1(Q^m_{p'})|\geq c_1 n]\ge 1- \exp\big(-2^{m}/m^{14}\big).\qedhere\]
\end{proof}

\begin{proof}[Proof of Theorem~\ref{thm: long-cycle-in-hypercube2}.]
Let $G=Q^m$. Note that $G_{p'}$ can be obtained by deleting edges in $G_p$ with probability $\ep/4$ independently. Let $\mathcal{A}$ be the event that $G_p$ is $(n/m^{16},n/m^{8})$-separable and $\mathcal{B}$ be the event that $|\mathcal{C}_1(G_{p'})|< n/m^{8}$. Assume that $\mathcal{A}$ occurs. Then we have a vertex subset $S$ with size at most $n/m^{16}$ such that every component of $G-S$ has at most $n/m^{8}$ vertices. If all edges between $S$ and $G-S$ are deleted when passing from $Q^m_p$ to $Q^m_{p'}$, then $\mathcal{B}$ happens. This deletion of all edges between $S$ and $G-S$ happens with probability at least $(\ep/4)^{|S|m}\geq(\ep/4)^{n/m^{15}}$. Hence, $\mathbb{P}[\mathcal{B}]\ge \mathbb{P}[\mathcal{A}]\cdot(\ep/4)^{n/m^{15}}$. However, $\mathbb{P}[\mathcal{B}] \le \exp(-n/m^{14})$ by Claim \ref{claim: largest component}.
 Thus we have  $\mathbb{P}[\mathcal{A}]\le \exp(-n/m^{14})\cdot(\ep/4)^{-n/m^{15}}=o(1)$.

By Corollary \ref{cor: separator}, w.h.p.\ $G_p$ has a subgraph $H$ such that $\abs{H}\ge n/m^{8}$ and $H$ has no separator with size at most $\abs{H}/(4m^{24})$. Thus we have $N_H(W)\geq \abs{H}/(4m^{24}) \geq  n/(4m^{32})$ for any $W\subseteq V(H)$ with $\abs{H}/3\leq \abs{W}\leq2\abs{H}/3$. Applying Theorem \ref{thm: long-cycle-in-locally-expander} we obtain that $H$, and so also $G_p$, has a cycle of length at least $n/(4m^{32})=2^{(1-o(1))m}$.
\end{proof}

\section{Concluding remarks}\label{sec:conc}
In this paper, we introduce the crux of a graph, corresponding to the order of the smallest dense patch of a graph, and study the `replacing average degree by crux' paradigm.  As a first example, we find in generic graphs cycles of length linear in the crux size and apply this result to address two conjectures of Long regarding long paths in subgraphs of hypercubes and Hamming graphs. As the
crux of a $C_4$-free graph is quadratic in its average degree, and the crux of a hypercube is exponential in its dimension, Theorems~\ref{thm:C4-free} and~\ref{thm: long-cycle-in-hypercube2}, on cycles in random subgraphs of $C_4$-free graphs and hypercube graphs are two more examples of this paradigm. It would be interesting to see more results of this form.

\section*{Acknowledgement}
We would like to thank Michael Krivelevich for bringing~\cite{KrLS} to our attention.
\medskip

\noindent\textbf{Note added before submission.} Theorem~\ref{thm: long-cycle-in-hypercube2} has been proved independently by Erde, Kang and Krivelevich~\cite{EKK} with a better cycle length $\Omega(\frac{2^m}{m^3\log^3m})$.




\begin{thebibliography}{99}
	\setlength{\parskip}{0pt}
	\setlength{\itemsep}{0pt plus 0.3ex}
	\footnotesize
	
	
    \bibitem{AKS1}
	M.~Ajtai, J.~Koml\'os, E. Szemer\'edi,
	\newblock The longest path in a random graph.
	\newblock \emph{ Combinatorica}, 1, (1981), 1--12.

    \bibitem{AKS}
	M.~Ajtai, J.~Koml\'os, E. Szemer\'edi,
	\newblock Largest random component of a $k$-cube.
	\newblock \emph{ Combinatorica}, 2, (1982), 1--7.

    \bibitem{AKS3}
	M.~Ajtai, J.~Koml\'os, E. Szemer\'edi,
	\newblock First occurence of Hamilton cycles in random graphs.
	\newblock \emph{ North-Holland Mathematics Studies}, 115(C), (1985), 173--178.

\bibitem{B1}
B.~Bollob{\'a}s,
\newblock The evolution of sparse graphs.
\newblock \emph{ Graph Theory and Combinatorics} (Cambridge 1983), (1984), 35--57.

\bibitem{BFF}
B.~Bollob\'as, T.~Fenner, A.~Frieze,
\newblock Long cycles in sparse random graphs.
\newblock \emph{ Graph Theory and Combinatorics} (Cambridge, 1983), (1984), 59--64.

	\bibitem{B-Th}
	B.~Bollob{\'a}s, A.~Thomason,
	\newblock Proof of a conjecture of {M}ader, {E}rd{\H o}s and {H}ajnal on
	topological complete subgraphs.
	\newblock \emph{ European Journal of Combinatorics}, 19, (1998), 883--887.

    \bibitem{CDGKO}
    P.~Condon, A.~Espuny~D\'iaz, A.~Gir\~ao, D.~K\"uhn, D.~Osthus,
    \newblock Hamiltonicity of random subgraphs of the hypercube.
    \newblock \emph{Proceedings of the 2021 ACM-SIAM Symposium on Discrete Algorithms (SODA).} Society for Industrial and Applied Mathematics, (2021), 889--898.


    \bibitem{Dirac}
	G.~A.~Dirac,
	\newblock Some theorems on abstract graphs.
	\newblock \emph{ Proceedings of the London Mathematical Society}, 2, (1952), 69--81.

    \bibitem{EJ}
S. Ehard, F. Joos,
\newblock Paths and cycles in random subgraphs of graphs with large minimum degree.
\newblock \emph{Electronic Journal of Combinatorics}, 25(2), (2018), P2.31.

\bibitem{EKK}
J. Erde, M. Kang, M. Krivelevich,
\newblock Expansion, long cycles, and complete minors in supercritical random subgraphs of the hypercube.
\newblock arXiv preprint, arXiv:2106.04249.

\bibitem{Erdos}
P. Erd\H{o}s, R. R\'enyi, V.T. S\'os,
\newblock On a problem of graph theory.
\newblock \emph{Studia Scientiarum Mathematicarum Hungarica}, 1, (1966), 215--235.


\bibitem{FKKL}
I.~Gil Fern\'andez, J.~Kim, Y.~Kim, H.~Liu,
\newblock Nested cycles with no geometric crossings.
\newblock \emph{Proceedings of the American Mathematical Society, Series B}, 9(03), (2022), 22--32.

    \bibitem{FK}
	L.~Friedman, M.~Krivelevich,
	\newblock Cycle lengths in expanding graphs.
	\newblock \emph{Combinatorica}, 41, (2021), 53--74.

\bibitem{Frieze}
A.M.~Frieze,
\newblock On large matchings and cycles in sparse random graphs.
\newblock \emph{ Discrete Mathematics}, 59(3), (1986), 243--256.


	
\bibitem{Harper}
L. H. Harper,
\newblock Optimal numberings and isoperimetric problems on graphs.
\newblock\emph{Journal of Combinatorial Theory}, 1, (1966), 385--393.


\bibitem{H-K-L}
J.~Haslegrave, J.~Kim, H.~Liu,
\newblock Extremal density for sparse minors and subdivisions.
\newblock \emph{International Mathematics Research Notices}, to appear.

\bibitem{JEMS}
R.~van der Hofstad, A.~Nachmias,
\newblock Hypercube percolation.
\newblock \emph{ Journal of the European Mathematical Society}, 19, (2017), 725--814.


\bibitem{survey:HLW}
S. Hoory, N. Linial, A. Wigderson,
\newblock Expander graphs and their applications.
\newblock \emph{Bulletin of the American Mathematical Society}, 43, (2006), 439--561.

	
\bibitem{IKL}
S.~Im, J.~Kim, Y.~Kim, H.~Liu,
\newblock Clique subdivisions in graphs without small dense subgraphs
\newblock \emph{preprint}.	
	
\bibitem{K-L}
P.~Keevash, E.~Long,
\newblock A stability result for the cube edge isoperimetric inequality.
\newblock \emph{Journal of Combinatorial Theory, Series A}, 155, (2018), 360--375.

\bibitem{KLShS17}
J.~Kim, H.~Liu, M.~Sharifzadeh, K.~Staden,
\newblock Proof of Koml\'os's conjecture on Hamiltonian subsets,
\newblock \emph{Proceedings of the London Mathematical Society}, 115 (5), (2017), 974--1013.

\bibitem{K-Sz-0}
J.~Koml{\'o}s, E.~Szemer{\'e}di,
\newblock  Limit distribution for the existence of Hamilton cycles in random graphs. \newblock \emph{Discrete Mathematics}, 43, (1983), 55--63.


\bibitem{K-Sz-1}
J.~Koml{\'o}s, E.~Szemer{\'e}di,
\newblock Topological cliques in graphs.
\newblock \emph{Combinatorics, Probability and Computing}, 3, (1994), 247--256.

\bibitem{K-Sz-2}
J.~Koml{\'o}s, E.~Szemer{\'e}di,
\newblock Topological cliques in graphs {II}.
\newblock \emph{Combinatorics, Probability and Computing}, 5, (1996), 79--90.

\bibitem{K-S-T}
T. K\H{o}v\'ari, V.T. S\'os, P. Tur\'an,
\newblock On a problem of K. Zarankiewicz.
\newblock \emph{Colloquium Mathematicum}, 3, (1954), 50--57.

\bibitem{Kri0}
M.~Krivelevich,
\newblock Long paths and Hamiltonicity in random graphs.
\newblock \emph{Random graphs, geometry and asymptotic structure}, 84, (2016), 1.

\bibitem{Kri}
M.~Krivelevich,
\newblock Long cycles in locally expanding graphs, with applications.
\newblock \emph{Combinatorica}, 39, (2019), 135--151.

\bibitem{Kri2}
M.~Krivelevich,
\newblock Expanders - how to find them, and what to find in them.
\newblock \emph{Surveys in Combinatorics}, 456, (2019), 115--142.

\bibitem{KrLS}
M.~Krivelevich, E.~Lubetzky, B.~Sudakov,
\newblock Asymptotics in percolation on high-girth expanders.
\newblock \emph{Random Structures $\&$ Algorithms}, (2020), 1--21.

\bibitem{K-L-S-0}
M.~Krivelevich, C.~Lee, B.~Sudakov,
\newblock Robust Hamiltonicity of Dirac graphs.
\newblock \emph{Transactions of the American Mathematical Society}, 366(6), (2014), 3095--3130.

\bibitem{K-L-S}
M.~Krivelevich, C.~Lee, B.~Sudakov,
\newblock Long paths and cycles in random subgraphs of graphs with large minimum degree.
\newblock \emph{Random Structures $\&$ Algorithms}, 46, (2015), 320--345.

\bibitem{KS14}
M. Krivelevich, W. Samotij,
\newblock Long paths and cycles in random subgraphs of $H$-free graphs.
\newblock \emph{Electronic Journal of Combinatorics}, 21(1), (2014), P1.30.

\bibitem{LS}
C.~Lee, B.~Sudakov,
\newblock Dirac's theorem for random graphs.
\newblock \emph{Random Structure $\&$ Algorithms}, 41(3) (2012), 293--305.


\bibitem{LM1}
H. Liu, R.H. Montgomery,
\newblock A proof of Mader's conjecture on large clique subdivisions in $C_4$-free graphs. \newblock \emph{Journal of the London Mathematical Society}, 95(1), (2017), 203--222.

\bibitem{LM2}
H. Liu, R.H. Montgomery,
\newblock A solution to Erd\H os and Hajnal's odd cycle problem.
\newblock arXiv preprint, arXiv:2010.15802.

\bibitem{LONG}
E.~Long,
\newblock Long paths and cycles in subgraphs of the cube.
\newblock \emph{Combinatorica}, 33, (2013), 395--428.

\bibitem{LWY}
H.~Liu, G.~Wang, D.~Yang,
\newblock Clique immersion in graphs without fixed bipartite graph.
\newblock arXiv preprint, arXiv:2011.10961.


\bibitem{Mader}
W.~Mader,
\newblock An extremal problem for subdivisions of $K_5^-$.
\newblock \emph{Journal of Graph Theory}, 30, (1999), 261--276.

\bibitem{Posa}
L.~P\'osa,
\newblock Hamiltonian circuits in random graphs.
\newblock \emph{Discrete Mathematics}, 14, (1976), 359--364.

\bibitem{Rei}
I. Reiman,
\newblock \"{U}ber ein {P}roblem von {K}. {Z}arankiewicz.
\newblock\emph{Acta Math. Acad. Sci. Hungar.}, 9, (1958), 269--273.

\bibitem{R}
O.~Riordan,
\newblock Long cycles in random subgraphs of graphs with large minimum degree.
\newblock \emph{Random Structure $\&$ Algorithms}, 45, (2014), 764--767.

\bibitem{Q-T-V}
R.~Squier, B.~Torrence, A.~Vogt,
\newblock The number of edges in a subgraph of a Hamming graph.
\newblock \emph{Applied Mathematics Letters}, 14, (2001), 701--705.



\end{thebibliography}
\end{document}